\documentclass[a4paper,12pt]{amsart}

\usepackage[colorlinks,citecolor = red, linkcolor=blue,hyperindex]{hyperref}
\usepackage{euscript,eufrak,verbatim}
\usepackage[psamsfonts]{amssymb}
\usepackage[usenames]{color}
\usepackage[curve]{xypic}
\usepackage{graphicx}
\usepackage{mathrsfs}

\usepackage{float}
\usepackage{bbm}
 \usepackage{color}
\newtheorem{thm}{Theorem}[section]
\newtheorem{lem}[thm]{Lemma}

\newtheorem{cor}[thm]{Corollary}

\newtheorem{rem}{Remark}[section]

\newtheorem*{prob*}{Problem}

\def\N{\mathbb{N}}
\def\Z{\mathbb{Z}}

\def\F{\mathbb{F}}

\def\PP{\mathbb{P}}

\def\1{\mathbbm{1}}
\def\an {\text{\, and \,}}

\def\rank{\mathrm{rank}}

\def\GL{\mathrm{GL}}

\def\modulo{\mathrm{\,mod\,}}

\def\XXint#1#2#3{{\setbox0=\hbox{$#1{#2#3}{\int}$ }
\vcenter{\hbox{$#2#3$ }}\kern-.6\wd0}}

\begin{document}

\title{Truncation of Haar random matrices in $\mathrm{GL}_N(\mathbb{Z}_m)$}

\date{\today}

\author
{Yanqi Qiu}
\address
{Yanqi QIU: CNRS, Institut de Math{\'e}matiques de Toulouse, Universit{\'e} Paul Sabatier, 118 Route de Narbonne, F-31062 Toulouse Cedex 9, France}

\email{yqi.qiu@gmail.com}

\thanks{This work is supported by the grant IDEX UNITI - ANR-11-IDEX-0002-02, financed by Programme ``Investissements d'Avenir'' of the Government of the French Republic managed by the French National Research Agency.}

\begin{abstract}
The asymptotic law of the truncated $S\times S$ random submatrix of a Haar random matrix in $\mathrm{GL}_N(\mathbb{Z}_m)$ as $N$ goes to infinity is obtained. The same result is also obtained when $\mathbb{Z}_m$ is replaced by any commutative compact local ring whose maximal ideal is topologically closed. 
\end{abstract}


\maketitle

\section{Introduction}
In the theory of random matrices, some particular attentions are payed recently to the asymptotic distributions of those truncated $S\times S$ upper-left corner of a large $N\times N$ random matrices from different matrix ensembles (CUE, COE, Haar Unitary Ensembles, Haar Orthogonal Ensembles), see \cite{TRU-1, TRU-2, TRU-3, TRU-4}. 

In the present paper, we consider the truncation of a Haar random matrix in $\GL_N(\Z_m)$ with $\Z_m = \Z/m\Z$. This research is motivated by its application in a forthcoming paper on the classification of ergodic measures on the space of infinite $p$-adic matrices, where the asymptotic law of a fixed size truncation of the Haar random matrix from the group of  $N\times N$ invertible matrices over the ring of $p$-adic integers is essentially used and is derived from a particular case of our main result, Theorem \ref{thm-main}. Remark that the ring of $p$-adic integers is isomorphic to the inverse limit of the rings $\Z_{p^n}$.

\section{Notation}
Fix a positive integer  $m\in\N$. Consider the ring $\Z_m : =\Z/m\Z$. Let $\Z_m^{\times}$ be the multiplicative group of invertible elements of the ring $\Z_m$. 
For any $N \in\N$, denote by $M_N(\Z_m)$ the matrix ring over $\Z_m$ and denote by $\GL_N(\Z_m)$ the finite  group of $N \times N$ invertible matrices over $\Z_m$.  Note that we have 
$$
\GL_N(\Z_m) = \Big\{A \in M_N(\Z_m) \Big| \det A  \in \Z_m^{\times}\Big\}.
$$ 
Let $\mathcal{U}_N(m)$ denote the uniform distribution on $M_N(\Z_m)$  and  let $\mu_N(m)$ denote the  uniform distribution on $\GL_N(\Z_m)$.  Note that $\mu_N(m)$ is the normalized Haar measure of the group $\GL_N(\Z_m)$.

The cardinality of any finite set $E$ is denoted by $|E|$.

\section{Main result}\label{sec-main}
Fix a positive integer $S\in\N$.  If $X$ is a $N\times N$ matrix (in what follows, the range of the coefficients of $X$ can vary), then we denote 
by $X[S]$  the truncated upper-left $S\times S$ corner of $X$, i.e.,  
\begin{align}\label{corner}
X[ S]: =(X_{ij})_{1 \le i, j \le S}
\end{align}
Let $X^{(N)}(m)$ be a random matrix sampled with respect to the normalized Haar measure of $\GL_N(\Z_m)$, that is, the probability distribution $\mathcal{L}(X^{(N)}(m))$  of the random matrix $X^{(N)}(m)$ satisfies 
$$
\mathcal{L}(X^{(N)}(m)) = \mu_N(m).
$$   
By adapting the notation \eqref{corner}, we denote by $X^{(N)}(m)[S] $ the truncated upper-left $S\times S$ corner of the random matrix $X^{(N)}(m)$, i.e.,  
$$
X^{(N)}(m)[S]: =\Big(X^{(N)}(m)_{ij}\Big)_{1 \le i, j \le S}
$$

\begin{thm}\label{thm-main}
The probability distribution $\mathcal{L}(X^{(N)}(m)[S])$ of the truncated random matrix $X^{(N)}(m)[S]$ converges weakly, as $N$ tends to infinity,  to the uniform distribution $\mathcal{U}_S(m)$ on $M_S(\Z_m)$. 
\end{thm}

For any  positive integer $u\in\N$, we write $Q_u: \Z\rightarrow \Z_u =\Z/u\Z$ the quotient map.  If $v$ is another positive integer such that  $u$ divides $v$, then since $v\Z \subset u \Z = \ker (Q_u)$, the map $Q_u$ induces in a unique way a map $Q_u^v: \Z_v\rightarrow \Z_u$. 
Note that the map $Q_u^v: \Z_v\rightarrow \Z_u$ is surjective and  
\begin{align}\label{preimage}
\Big| (Q_u^v)^{-1}(x) \Big|= \frac{v}{u}, \forall x \in\Z_u,
\end{align}
that is, for each element $ x \in \Z_u$, the cardinality of  the pre-image of $x$ is $v/u$. 

By slightly abusing the notation, for any matrix $A = (a_{ij})_{1\le i, j \le N}$ in $M_N(\Z)$, we set
$$ Q_{u} (A) : = (Q_{u}(a_{ij}))_{1 \le i, j \le N}.$$ Similarly, for any matrix $B = (b_{ij})_{1\le i, j \le N}$  in $M_N(\Z_v)$, we set $$ Q_{u}^v (B) : = (Q_{u}^v(b_{ij}))_{1 \le i, j \le N}.$$

  By the prime factorization theorem, we may write in a unique way  
  \begin{align}\label{factor}
  m = p_1^{r_1} \cdots p_s^{r_s},
  \end{align}
  where $p_1,\cdots, p_s$ are distinct prime numbers and $r_1, \cdots, r_s $ are positive integers.  By the Chinese remainder theorem, we have an isomorphism of the following two rings:
\begin{align}
\Z_m \simeq \Z_{p_1^{r_1}} \oplus  \cdots \oplus  \Z_{p_s^{r_s}}.
\end{align}
A natural isomorphism is provided by the map $\phi: \Z_m \longrightarrow \Z_{p_1^{r_1}} \oplus  \cdots \oplus  \Z_{p_s^{r_s}}$  defined by
\begin{align}\label{iso}
\phi(x) = (Q^m_{p_1^{r_1}}(x), \cdots, Q^m_{p_s^{r_s}}(x) )  \quad (\forall x \in \Z_m).
\end{align}

{\flushleft \it  Simple case:} Let us first assume that in the factorization \eqref{factor}, we have $s =1$. For simplifying the notation, let us write $m = p^{r}$. 

Writing $\F_p$ for the finite field $\Z/p\Z$. We have the following characterization of $\GL_N(\Z_{p^r})$.  
\begin{thm}[{\cite[Theorem 3.6]{Aut-abelian}}]
A matrix $M$ is in $\GL_N(\Z_{p^r})$  if and only if $Q_{p}^{p^r}(M) \in\GL_N(\F_p)$. 
\end{thm}

Given a matrix $W \in M_S(\Z_{p^r})$ such that $Q^{p^r}_p(W) \in\GL_S(\F_p)$, then a moment of thinking  allows us to write 
\begin{align}
\Big|\big\{X \in \GL_N (\F_p): X(S)=  Q^{p^r}_p(W) \big\} \Big| =  p^{S(N-S)} \cdot \prod_{j = 0}^{N-S-1} (p^N - p^{S+j}),
\end{align}
where $p^{S(N-S)}$ the number of choices of $(X_{ij})_{1 \le i \le S, \, S+1 \le j \le N} $ with coefficients in $\F_p$ and $\prod_{j = 0}^{N-S-1} (p^N - p^{S+j})$ is the number of choices of $(X_{ij})_{S+1 \le i \le N, \, 1 \le j \le N} $. 

It follows that, for any matrix $W \in M_S(\Z_{p^r})$,  we have 
\begin{align}\label{upper}
\Big|\big\{X \in \GL_N (\F_p): X(S)=  Q^{p^r}_p(W) \big\} \Big| \le   p^{S(N-S)} \prod_{j = 0}^{N-S-1} (p^N - p^{S+j}).
\end{align}
We also have for any matrix $W \in M_S(\Z_{p^r})$, 
\begin{align}\label{lower}
\Big|\big\{X \in \GL_N (\F_p): X(S)=  Q^{p^r}_p(W) \big\} \Big| \ge \prod_{i = 0}^{S-1}  (p^{N-S} - p^{i}) \prod_{j = 0}^{N-S-1} (p^N - p^{S+j}),
\end{align}
where $\prod_{i = 0}^{S-1}  (p^{N-S} - p^{i})$ is the number of choices of $(X_{ij})_{1 \le i \le S, \, S+1 \le j \le N} $ with coefficients in $\F_p$ such that
$$
\rank \Big[(X_{ij})_{1 \le i \le S, \, S+1 \le j \le N}\Big] = S.
$$

Recall that $X^{(N)}(m)$ is a random matrix sampled with respect to the Haar measure of $\GL_N(\Z_m) = \GL_N(\Z_{p^r})$. By combining \eqref{preimage}, \eqref{upper} and \eqref{lower}, we see that the cardinality 
$$
n_N(W): = \Big|\big\{X \in \GL_N (\Z_{p^r}): X(S)=  W \big\} \Big|
$$
satisfies the relation 
\begin{align*}
  & p^{r-1} \prod_{i = 0}^{S-1}  (p^{N-S} - p^{i}) \prod_{j = 0}^{N-S-1} (p^N - p^{S+j}) \le n_N(W)  
\\
& \le  p^{r-1}  p^{S(N-S)} \prod_{j = 0}^{N-S-1} (p^N - p^{S+j}).
\end{align*}
As a consequence, for any $W_1, W_2 \in M_S(\Z_{p^r})$, the following relation holds:
\begin{align}\label{estimate}
\frac{\prod_{i = 0}^{S-1}  (p^{N-S} - p^{i})}{ p^{S(N-S)}} \le \frac{\PP(X^{(N)}(m)[S] = W_1)}{\PP( X^{(N)}(m)[S] = W_1)} \le \frac{ p^{S(N-S)}}{\prod_{i = 0}^{S-1}  (p^{N-S} - p^{i})}.
 \end{align}
 Hence we get 
 \begin{align}\label{limit-1}
 \lim_{N\to\infty}\frac{\PP(X^{(N)}(m)[S] = W_1)}{\PP(X^{(N)}(m)[S]  = W_1)}  = 1.
 \end{align}
Since the set $M_S(\Z_m)$ is finite, the above equality \eqref{limit-1} implies that  $\mathcal{L}(X^{(N)}(m)[S])$ converges weakly, as $N$ tends to infinity,  to the uniform distribution $\mathcal{U}_S(m)$ on $M_S(\Z_m)$.

{\flushleft \it General case:}
It is clear that, for any $N\in\N$, the isomorphism $\phi$ defined in \eqref{iso} induces in a natural way a ring isomorphism:
\begin{align}
\phi_N: M_N (\Z_m) \xrightarrow{\quad \simeq \quad } M_N (\Z_{p_1^{r_1}}) \oplus  \cdots \oplus  M_N(\Z_{p_s^{r_s}}).
\end{align}
The restriction of $\phi_N$ on $\GL_N(\Z_m)$ induces a group isomorphism:
\begin{align}
\phi_N: \GL_N (\Z_m) \xrightarrow{\quad \simeq \quad } \GL_N (\Z_{p_1^{r_1}}) \oplus  \cdots \oplus  \GL_N(\Z_{p_s^{r_s}}).
\end{align}
Obviously, we have 
\begin{align}
(\phi_N)_{*}  (\mathcal{U}_N(m)) &= \mathcal{U}_N(p_1^{r_1}) \otimes  \cdots \otimes  \mathcal{U}_N(p_s^{r_s})
\end{align}
and
 \begin{align}
 (\phi_N)_{*}  (\mu_N(m)) & = \mu_N(p_1^{r_1}) \otimes  \cdots \otimes  \mu_N(p_s^{r_s}).
\end{align}
In particular, if $X^{(N)}(p_1^{r_1}), \cdots, X^{(N)}(p_s^{r_s})$ are independent Haar random matrices in $\GL_N (\Z_{p_1^{r_1}}), \cdots, \GL_N (\Z_{p_s^{r_s}})$ respectively, then the random matrix 
$$
\phi_N^{-1} (X^{(N)}(p_1^{r_1}) \oplus \cdots \oplus X^{(N)}(p_s^{r_s}) )
$$
is a Haar random matrix in $\GL_N (\Z_m)$.  Moreover, we have 
$$
\phi_N^{-1} (X^{(N)}(p_1^{r_1}) \oplus \cdots \oplus X^{(N)}(p_s^{r_s}) )[S]  = \phi_S^{-1} (X^{(N)}(p_1^{r_1}) [S]   \oplus \cdots \oplus X^{(N)}(p_s^{r_s})  [S]  ).
$$
Hence $X^{(N)}(m)[S]$ and $\phi_S^{-1} (X^{(N)}(p_1^{r_1}) [S]   \oplus \cdots \oplus X^{(N)}(p_s^{r_s})  [S] )$ are identically distributed.  By the previous result, we know that for any $i = 1, \cdots, s$, the law of $X^{(N)}(p_i^{r_i}) [S]$ converges weakly to the uniform distribution $\mathcal{U}_S(p_i^{r_i})$ on $M_S (\Z_{p_i^{r_i}})$. It follows that the law of $X^{(N)}(m)[S]$ converges weakly to 
$$
(\phi_S^{-1})_{*}  (\mathcal{U}_S(p_1^{r_1}) \otimes  \cdots \otimes  \mathcal{U}_S(p_s^{r_s})) = \mathcal{U}_S(m).
$$
We thus complete the proof of Theorem \ref{thm-main}. 
\section{A generalization}

Let $\F_q$ denote the finite field with cardinality $q = p^n$. Consider the Haar random matrix $Z^{(N)}$ in $\GL_N(\F_q)$. Then we have 
\begin{thm}\label{thm-2}
The probability distribution $\mathcal{L}(Z^{(N)}[S])$ of the truncated random matrix $Z^{(N)}[S]$ converges weakly, as $N$ tends to infinity,  to the uniform distribution on $M_S(\F_q)$. 
\end{thm}

\begin{proof}
By combinatoric arguments, we have a similar estimate as \eqref{estimate} and the proof of Theorem \ref{thm-2} then follows immediately. Here we omit the details. 
\end{proof}

Let $(\mathscr{A}, +, \cdot)$ be a topological commutative  ring with identity which is compact, thus by assumption, the two operations $+, \cdot: \mathscr{A} \times \mathscr{A} \rightarrow \mathscr{A}$ are both  continuous. Assume also that  $\mathscr{A}$ is a local ring.  Recall that  by local ring, we mean that $\mathscr{A}$ admits a unique maximal ideal. Let us denote the maximal ideal of $\mathscr{A}$ by $\mathfrak{m}$.  If we denote by $\mathscr{A}^{\times} $  the multiplicative group of the $\mathscr{A}$, then we have  $\mathfrak{m} =\mathscr{A}\setminus \mathscr{A}^{\times}.$
Moreover, let us assume that $\mathfrak{m}$ is closed. 

\begin{rem}
If $m = p_1^{r_1}\cdots p_s^{r_s}$ with $s \ge 2$, then the ring $\Z_m$ is not local. Thus the results in \S \ref{sec-main} is not a particular case of Theorem \ref{thm-3}. 
\end{rem}

 Denote by $\nu_\mathscr{A}$ the normalized Haar measure on the compact additive group $(\mathscr{A}, +)$.  

\begin{lem}[{\cite[Lemma 3]{Warner}}]
The quotient ring $\mathscr{A}/\mathfrak{m}$ is a finite field. 
\end{lem}

As a consequence, there exists a positive integer $q = p^n$ with $p$   a prime number  and $n$ a positive integer,  such that $|\mathscr{A}/\mathfrak{m}| = q$ and $\mathscr{A}/\mathfrak{m} \simeq \F_q$. Let $\{a_{i}: i = 0, \cdots, q-1\}$ be a subset of $\mathscr{A}$ which forms a complete set of representatives of $\mathscr{A}/\mathfrak{m}$, assume moreover that $a_0 =0\in \mathscr{A}$. 

From now on,  as a set, we will identify $\{a_{i}: i = 0, \cdots, q-1\}$ with $\F_q$. For instance, under this identification, we may write
$$
\mathscr{A} = \bigsqcup_{i = 0}^{q-1} (a_i + \mathfrak{m}) = \bigsqcup_{x \in\F_q} (x + \mathfrak{m}),
$$
we also identify the following subset of $M_N(\mathscr{A})$: 
\begin{align}\label{identification}
\Big\{ X  = (X_{ij})_{1\le i, j \le N} \Big| X_{ij}  \in \{a_{i}: 0\le i \le q-1\}, \det X \in  \mathscr{A}^{\times}   \Big\}
\end{align}
with $\GL_N(\F_q)$. 

Since $\mathscr{A}^{\times} $  is closed, indeed, the  group of invertible matrices over $\mathscr{A}$:
$$
\GL_N(\mathscr{A}) = \Big\{A \in M_N(\mathscr{A}) \Big| \det A  \in \mathscr{A}^{\times}\Big\},
$$
as a closed subset of $M_N(\mathscr{A})$, is compact. As a consequence, we may speak of Haar random matrix in $\GL_N(\mathscr{A})$, let  $Y^{(N)}$ be such a random matrix. We would like to study the asymptotic law of the truncated random matrix $Y^{(N)}[S]$ as $N$ goes to infinity.

\begin{lem}\label{lem-disjoint}
We have 
\begin{align}\label{disjoint}
\GL_N(\mathscr{A}) = \bigsqcup_{X \in \GL_N(\F_q)}  (X+ M_N(\mathfrak{m})), 
\end{align}
where we identify $\GL_N(\F_q)$ with the set given by \eqref{identification}.
\end{lem}

\begin{proof}
It is easy to see that for any $X \in\GL_N(\F_q)$ and any $X'\in M_N(\mathfrak{m}) $, we have 
$$
\det (X + X') \equiv \det X (\modulo \mathfrak{m}),
$$
hence $\det (X+X') \in \mathscr{A}^{\times}$. This implies that the set on the right hand side of \eqref{disjoint} is contained in $\GL_N(\mathscr{A})$.  Conversely, an element $A \in \GL_N(\mathscr{A}) \subset M_N(\mathscr{A})$ corresponds naturally to a matrix $X_A \in M_N(\mathscr{A})$ all of whose coefficients are in $\F_q$ (identified with $\{a_i: 0\le i \le q-1\}$) such that 
$$
A \equiv X_A  (\modulo \mathfrak{m}) \an \det A \equiv \det X_A (\modulo \mathfrak{m}).
$$
As a consequence, $\det X_A \in \mathscr{A}^{\times}$ and hence $X_A \in\GL_N(\F_q)$. This shows that $\GL_N(\mathscr{A})$ is contained in the set on the right hand side of \eqref{disjoint}.  

Finally, by the definition of the set $\GL_N(\F_q)$ in \eqref{disjoint}, it is clear that all the subsets $X+ M_N(\mathfrak{m}), X \in \GL_N(\F_q)$ are disjoint. 
\end{proof}

As an immediate consequence of Lemma \ref{lem-disjoint}, we have the following corollary. First recall that we have identified $\GL_N(\F_q)$ with the set \eqref{identification}, hence the random matrix $Z^{(N)}$ may be considered as a random matrix sampled uniformly from the set  \eqref{identification}. Note that $M_N(\mathfrak{m}) \simeq \mathfrak{m}^{N\times N}$ is equipped with the uniform probability 
\begin{align}\label{unif}
(q^{-1} \nu_\mathscr{A}|_{\mathfrak{m}})^{\otimes (N\times N)}. 
\end{align}

\begin{cor}\label{cor-in}
Assume that we are given a random matrix  $U^{(N)}$ sampled uniformly from $M_N(\mathfrak{m})$, which is independent from the random matrix $Z^{(N)}$. The the random matrix 
$$
Z^{(N)} + U^{(N)}
$$
is a Haar random matrix in $\GL_N(\mathscr{A})$. 
\end{cor}

Note that the distributions of the two random matrices $U^{(S)}$ and $U^{(N)}[S]$ coincide.

\begin{thm}\label{thm-3}
The probability distribution $\mathcal{L}(Y^{(N)}[S])$ of the truncated random matrix $Y^{(N)}[S]$ converges weakly, as $N$ tends to infinity,  to the uniform distribution $\nu_\mathscr{A}^{\otimes (S\times S)}$ on $M_S(\mathscr{A})$.
\end{thm}

\begin{proof}
By Corollary \ref{cor-in}, the random matrices $Y^{(N)}[S]$ and  $Z^{(N)}[S] + U^{(N)}[S]$ are identically distributed. Now by Theorem \ref{thm-2}, the probability distribution $\mathcal{L}(Z^{(N)}[S])$ converges weakly, as $N$ goes to infinity, to the uniform distribution on $M_S(\F_q)$, hence the probability distribution $\mathcal{L}(Y^{(N)} [S]) = \mathcal{L}(Z^{(N)} [S] + U^{(N)} [S]  )$ converges weakly, as $N$ goes to infinity, to the probability distribution of the random matrix 
$$
V^{(S)} + U^{(S)},
$$
where $V^{(S)}$ and $U^{(S)}$ are independent, $V^{(S)}$ is sampled uniformly from $M_S(\F_q)$ and $U^{(S)}$ is sampled uniformly from $M_N(\mathfrak{m})$.   We complete the proof of Theorem \ref{thm-3} by noting that $V^{(S)} + U^{(S)}$ is uniformly distributed on $M_S(\mathscr{A})$.
\end{proof}


\begin{thebibliography}{1}

\bibitem{TRU-4}
Zhishan Dong, Tiefeng Jiang, and Danning Li.
\newblock Circular law and arc law for truncation of random unitary matrix.
\newblock {\em J. Math. Phys.}, 53(1):013301, 14, 2012.

\bibitem{TRU-3}
Yan~V. Fyodorov and Boris~A. Khoruzhenko.
\newblock A few remarks on colour-flavour transformations, truncations of
  random unitary matrices, {B}erezin reproducing kernels and {S}elberg-type
  integrals.
\newblock {\em J. Phys. A}, 40(4):669--699, 2007.

\bibitem{Aut-abelian}
Christopher~J. Hillar and Darren~L. Rhea.
\newblock Automorphisms of finite abelian groups.
\newblock {\em Amer. Math. Monthly}, 114(10):917--923, 2007.

\bibitem{TRU-2}
J.~Novak.
\newblock Truncations of random unitary matrices and {Y}oung tableaux.
\newblock {\em Electron. J. Combin.}, 14(1):Research Paper 21, 12, 2007.

\bibitem{Warner}
Seth Warner.
\newblock Compact noetherian rings.
\newblock {\em Math. Ann.}, 141:161--170, 1960.

\bibitem{TRU-1}
Karol {\.Z}yczkowski and Hans-J{\"u}rgen Sommers.
\newblock Truncations of random unitary matrices.
\newblock {\em J. Phys. A}, 33(10):2045--2057, 2000.

\end{thebibliography}

\end{document}